\documentclass[11pt]{article}

\usepackage[margin=3cm]{geometry}

\usepackage[T1]{fontenc}
\usepackage[utf8]{inputenc}

\usepackage{fourier} %pour le symbole \danger
\usepackage{soul} %pour souligner \ul avec respect des retours à la ligne et pour barrer avec \st
\usepackage{enumerate} %pour les \begin{enumerate}[(i)] par exemple

\usepackage{amsfonts,amssymb,amsmath,amsthm}
\usepackage{bbm}
\usepackage{pgf}

\usepackage{hyperref}
\usepackage{cancel}

\usepackage{xcolor}
\definecolor{darkblue}{rgb}{0,0,0.5}
\usepackage{transparent}
\usepackage{graphicx}
\newcommand{\executeiffilenewer}[3]{%
  \ifnum\pdfstrcmp{\pdffilemoddate{#1}}%
  {\pdffilemoddate{#2}}>0%
  {\immediate\write18{#3}}\fi%
}

\numberwithin{equation}{section}
\def\PP{\mathbb{P}}

\def\RR{\mathbb{R}}

\def\ZZ{\mathbb{Z}}

\def\EE{\mathbb{E}}
\def\11{\mathbbm{1}}

\def\E{\mathbb{E}}

\def\R{\mathbb{R}}

\def\N{\mathbb{N}}

\def\d{\partial}
\def\Z{\mathbb{Z}}

\def\cL{{\cal L}}

\newtheorem{thm}{Theorem}[section]

\newtheorem{cor}[thm]{Corollary}

\newtheorem{prop}[thm]{Proposition}

\theoremstyle{remark}
\newtheorem{rem}{Remark}

\newcommand{\vertiii}[1]{{\left\vert\kern-0.25ex\left\vert\kern-0.25ex\left\vert #1 
    \right\vert\kern-0.25ex\right\vert\kern-0.25ex\right\vert}}
\begin{document}

\title{Practical criteria for $R$-positive recurrence of unbounded semigroups}

\author{Nicolas Champagnat$^{1}$, Denis Villemonais$^{1}$}
\footnotetext[1]{Université de Lorraine, CNRS, Inria, IECL, UMR 7502, F-54000 Nancy, France\\
  E-mail: Nicolas.Champagnat@inria.fr, Denis.Villemonais@univ-lorraine.fr}

\maketitle

\begin{abstract}
  The goal of this note is to show how recent results on the
  theory of quasi-stationary distributions allow us to deduce
   general criteria for the geometric convergence of
  normalized unbounded semigroups.
\end{abstract}


%\medskip\noindent\textit{2010 Mathematics Subject Classification.} Primary: 

\section{Introduction}
\label{sec:intro}

Let $E$ be a measurable space and $(P_n,n\in\ZZ_+)$ be a positive semigroup of operators on the space $L^\infty(\psi_1)$ to itself,
where $\psi_1:E\rightarrow(0,+\infty)$ is measurable and $L^\infty(\psi_1)$ is the set of measurable $f:E\rightarrow\RR$ such that
$|f|/\psi_1$ is bounded, endowed with the norm $\|f\|_{\psi_1}=\||f|/\psi_1\|_\infty$. We define the dual action of $(P_n,n\in\ZZ_+)$ on non-negative measures $\mu$ on $E$ such that
$\mu(\psi_1)<+\infty$ as
\begin{equation}
  \label{eq:left-action}
  \mu P_n f=\int_E P_n f(x)\mu({\rm d}x).
\end{equation}
Our aim is to provide sufficient conditions for the existence of $\theta_0>0$ such that $(\theta_0^{-n}P_n)_{n\in\N}$ converges
geometrically toward a non-trivial limit.

In this setting, given $c$ such that $P_1\psi_1\leq c\psi_1$, the operators $Q_n=\frac{P_n (\cdot\psi_1)}{c^n\psi_1}$ defines a
sub-Markov semigroup corresponding to a stochastic process with killing. The asymptotic behavior of such semigroups is the subject of
the theory of quasi-stationary distributions based on various tools, including the theory of $R$-recurrent Markov
chains~\cite{TuominenTweedie1979,Nummelin1984,NiemiNummelin1986,FerrariKestenEtAl1996}, spectral theoretic results (e.g.\
Krein-Rutman theorem~\cite{ColletMartinezEtAl2013}, spectral theory of symetric operators~\cite{CCLMMS09,KolbSteinsaltz2012}, or
other general criteria of convergence of normalized semigroups such as the work of Birkhoff~\cite{birkhoff-57} and its extensions)
and Doeblin's conditions and Foster-Lyapunov criteria~\cite{ChampagnatVillemonais2016b,ChampagnatVillemonais2017b}. In this note, we
apply the results of~\cite{ChampagnatVillemonais2017b} to the semigroup $(Q_n,n\in\ZZ_+)$ to give a necessary and sufficient
condition for the existence of a nonnegative eigenfunction $\eta$ of $P_1$ with eigenvalue $\theta_0$ and the geometric convergence
of $\theta_0^{-n} P_n$. We also extend these results to continuous-time semigroups. In particular, our results provide practical
criteria for the general theory of $R$-positive recurrence of unbounded semigroups as developed in~\cite[Section 6.2]{Nummelin1984}
and~\cite{NiemiNummelin1986}. The notion of $R$-positive recurrence has strong implications for the study of penalized Markov
processes~\cite{DelMoral2004,DelMoral2013}, of the long time behaviour of Markov branching processes (see for
instance~\cite{IkedaNagasawaEtAl1968,IkedaNagasawaEtAl1968a,IkedaNagasawaEtAl1969,BigginsKyprianou2004,Jagers1989,Cloez2017,BertoinWatson2018,Bertoin2018,BertoinWatson2019}),
of non-conservative PDEs (see e.g.~\cite{BansayeCloezEtAl2018,BansayeCloezEtAl2019} and references therein) and of measure-valued
P\'olya processes and reinforced processes~\cite{MaillerVillemonais2018}.

The recent article~\cite{BansayeCloezEtAl2019} proposes similar criteria for $R$-positive recurrence of continous-time semigroups
with nice applications to growth-fragmentation equations. The extent of our results and approaches sensibly differ. Concerning the
results, our criteria apply to a larger class of semigroups including non-irreducible ones (see Remark~\ref{rem:BCGM} below). Concerning the
approaches, the authors of~\cite{BansayeCloezEtAl2019} make use of tools developed in the proofs of~\cite{ChampagnatVillemonais2017b}
adapted to the semigroup setting. We show here how these $R$-positivity criteria can be directly derived as corollaries of the results of~\cite{ChampagnatVillemonais2017b}, applied to the sub-Markov semigroup $(Q_n,n\in\Z_+)$. This approach also has the advantage to allow one to deduce with little
extra effort sufficient criteria for the convergence of unbounded semigroups from the abundant theory of sub-Markov processes (cf.
e.g.~\cite{ColletMartinezEtAl2013,ColletMartinezEtAl2011,Velleret2018,FerreRoussetEtAl2018,KolbSteinsaltz2012,HeningKolb2019}). Note
that a similar approach has been used in~\cite{BertoinWatson2018} to describe the asymptotic behaviour of the growth-fragmentation
equation using Krein-Rutman theorem and other criteria for $R$-positivity. Finally, the authors of~\cite{BansayeCloezEtAl2019} also establish
a counterpart assuming the existence of a positive eigenfunction of the semigroup and using the approach
of~\cite{ChampagnatVillemonais2016b}. In Theorem~\ref{thm:reciproque}, we extend this counterpart by allowing the eigenfunction to
vanish and exhibit the link with the classical theory of $V$-ergodicity~\cite{MeynTweedie2009,DoucMoulinesEtAl2018}.

Section~\ref{sec:main-result} is devoted to the statement and the proof of our main results. In Section~\ref{sec:applications}, we
provide two applications of these general results to penalized semigroups associated to perturbed (discrete-time) dynamical systems
(Subsection~\ref{sec:perturbedDyn}) and diffusion processes (Subsection~\ref{sec:diffusionProc}).

\section{Main result}
\label{sec:main-result}

We first introduce the assumptions on which our results are based. We state them following the same structure as Assumption~(E)
in~\cite{ChampagnatVillemonais2017b} to emphasize their similarity.

\medskip\noindent\textbf{Condition (G).} There exist positive real
constants $\theta_1,\theta_2,c_1,c_2,c_3$, an integer $n_1\geq 1$, two
functions $\psi_1:E\rightarrow(0,+\infty)$, $\psi_2:E\rightarrow \R_+$ and a probability measure
$\nu$ on a measurable subset $K$ of $E$ such that
\begin{itemize}
\item[(G1)] \textit{(Local Dobrushin coefficient).} $\forall x\in K$ and all measurable $A\subset K$,
  \begin{align*}
    P_{n_1}(\psi_1\11_A)(x)\geq c_1 \nu(A)\psi_1(x).
  \end{align*}
\item[(G2)] \textit{(Global Lyapunov criterion).} We have $\theta_1<\theta_2$ and
  \begin{align*}
    &\inf_{x\in K} \psi_2(x)/\psi_1(x)>0,\ \sup_{x\in E}\psi_2(x)/\psi_1(x)\leq 1,\\
    &P_{1}\psi_1(x)\leq \theta_1\psi_1(x)+c_2\11_K(x)\psi_1(x),\ \forall x\in E,\\
    &P_{1}\psi_2(x)\geq \theta_2\psi_2(x),\ \forall x\in E.
  \end{align*}
\item[(G3)] \textit{(Local Harnack inequality).} We have
  \begin{align*}
    \sup_{n\in \Z_+}\frac{\sup_{y\in K} P_n\psi_1(y)/\psi_1(y)}{\inf_{y\in K} P_n\psi_1(y)/\psi_1(y)}\leq c_3.
  \end{align*}
\item[(G4)] \textit{(Aperiodicity).} For all $x\in K$, there exists $n_4(x)$ such that for all $n\geq n_4(x)$,
  \[
    P_n(\11_K{\psi_1})>0.
  \]
\end{itemize}
\medskip

\begin{thm}
  \label{thm:main}
  Assume that Condition~(G) holds true. Then there exist a positive measure $\nu_P$ on $E$ such that $\nu_P(\psi_1)=1$
  and $\nu_P(\psi_2)>0$, and two constants $C<+\infty$ and $\alpha\in(0,1)$ such that, for all measurable functions
  $f:E\rightarrow\R$ satisfying $|f|\leq \psi_1$ and all positive measures $\mu$ on $E$ such that $\mu(\psi_1)<+\infty$ and
  $\mu(\psi_2)>0$,
  \begin{align}
    \label{equation1}
    \left|\frac{\mu P_n f}{\mu P_n \psi_1}-\nu_P(f)\right|\leq C\alpha^n \frac{\mu(\psi_1)}{\mu(\psi_2)},\quad\forall n\in\ZZ_+.
  \end{align}
  In addition, there exist $\theta_0>0$ such that $\nu_P P_n=\theta_0^n\nu_P$ and a function $\eta:E\rightarrow\R_+$ such that
  $\theta_0^{-n}P_n\psi_1$ converges uniformly and geometrically toward $\eta$ in $L^\infty(\psi_1)$ and such that
  $P_1\eta=\theta_0\eta$ and $\nu_P(\eta)=\nu_P(\psi_1)=1$. Moreover, there exist two constants $C'>0$ and $\beta\in(0,1)$ such that,
  for all measurable functions $f:E\rightarrow\R$ satisfying $|f|\leq \psi_1$ and all positive measures $\mu$ on $E$ such that
  $\mu(\psi_1)<+\infty$,
  \begin{align}
    \label{equation2}
    \left|\theta_0^{-n}\mu P_n f-\mu(\eta)\nu_P(f)\right|\leq C'\beta^n \mu(\psi_1).
  \end{align}
\end{thm}

\begin{rem}
  \label{rem:comments-hyp-G}
  Note that (G2) implies that $P_n\psi_1\leq c P_n\psi_2$ on $K$ for all
  $n\geq 0$ and some constant $c>0$
  (see~\cite[Lemma~9.6]{ChampagnatVillemonais2017b}).  Hence we have, for all $x\in K$,
  \[
  P_n\psi_1(x)/\psi_1(x)\leq c\,P_n\psi_2(x)/\psi_1(x)\leq c\,P_n\psi_2(x)/\psi_2(x)
  \]
  and
  \[
    P_n\psi_2(x)/\psi_2(x)\leq P_n\psi_1(x)/\psi_2(x)\leq
      \sup_K \frac{\psi_1}{\psi_2}
    \,P_n\psi_1(x)/\psi_1(x).
  \]
  Therefore, replacing $\psi_1$ by $\psi_2$ in (G1) and/or (G3) give equivalent versions of Condition~(G). 
\end{rem}

\begin{proof}
  Assumption~(G2) implies that $P_1\psi_1\leq (\theta_1+c_2)\psi_1$,
  so that $Q_1f:=\frac{P_1(f\psi_1)}{(\theta_1+c_2)\psi_1}$ defines a
  submarkovian kernel generating the semigroup $(Q_n)_{n\in\N}$
  defined by
  \[
    Q_n (f)=\frac{P_n (f\,\psi_1)}{(\theta_1+c_2)^n\psi_1},\ \forall n\geq 0,\ \|f\|_\infty\leq 1.
  \]
  It is straightforward to check that this semigroup satisfies
  conditions (E1-E4) in~\cite{ChampagnatVillemonais2017b} with
  $\varphi_1=1$ and $\varphi_2=\psi_2/\psi_1$, using
  $\theta_1/(\theta_1+c_2)$ in place of $\theta_1$,
  $\theta_2/(\theta_1+c_2)$ in place of $\theta_2$ and
  $c_1/(\theta_1+c_2)^{n_1}$ in place of $c_1$. Using Theorem~2.1 in
  this reference applied to $Q_n$, we deduce that there exist
  constants $C>0,\alpha\in(0,1)$ and a probability measure $\nu_{QSD}$
  on $E$ such that, for all bounded measurable functions
  $g:E\rightarrow\R$ and all probability measures $\upsilon$ such that
  $\upsilon(\varphi_2)>0$,
  \begin{align*}
    \left|\frac{\upsilon Q_n g}{\upsilon Q_n \11}-\nu_{QSD}(g)\right|\leq C\alpha^n \frac{\|g\|_\infty}{\upsilon(\varphi_2)}.
  \end{align*}  
  Setting $\nu_P(dx)=\frac{1}{\psi_1(x)}\nu_{QSD}(dx)$,
  $\mu(dx)=\frac{1}{\psi_1(x)}\upsilon(dx)$ and $f=g\,\psi_1$, one
  obtains~\eqref{equation1}. Similarly, Theorem~2.5 of~\cite{ChampagnatVillemonais2017b} for $Q_n$
  states that there exist $\theta_Q>0$ such that
  $\nu_{QSD} Q_n=\theta_Q^n\nu_{QSD}$ and a function
  $\eta_Q:E\rightarrow\R_+$ such that $\theta_Q^{-n}Q_n\11$
  converges uniformly and geometrically toward $\eta_Q$ in $L^\infty$
  and such that $Q_1\eta_Q=\theta_Q\eta_Q$. Setting
  $\eta=\eta_Q\psi_1$ and $\theta_0=\theta_Q(\theta_1+c_2)$ gives the result on geometric convergence of $\theta_0^{-n}P_n\psi_1$ to
  $\eta$ in $L^\infty(\psi_1)$.

  It remains to prove~\eqref{equation2}. Note that it is sufficient to prove it for any $\mu=\delta_x$. If $\eta(x)=0$, this is implied by the above
  geometric convergence. If $\eta(x)>0$, then $\eta_Q(x)>0$ and the convergence of \cite[Theorem\,2.7]{ChampagnatVillemonais2017b}
  applied to $Q_n$ implies that there exists $C'<+\infty$ and $\tilde\alpha\in(0,1)$ such that, for all measurable
  $g:E\rightarrow\RR$ satisfying $|g|\leq 1/\eta_Q$,
  \[
    \left|\theta_Q^{-n} \frac{Q_n (g\eta_Q)(x)}{\eta_Q(x)}-\nu_{QSD}(g\eta_Q)\right|\leq C'\tilde\alpha^n \frac{1}{\eta_Q(x)}. 
  \]
  Multiplying both sides by $\eta_Q(x)\psi_1(x)$ and setting $f=g\eta_Q\psi_1$ ends the proof of~\eqref{equation2}.
\end{proof}

Whether Assumption~(G) is necessary for~\eqref{equation1} is still an open problem up to our knowldge. However, if one assumes that
there exists a positive eigenfunction $\eta$ such that~\eqref{equation2} holds true, then one recovers easily Assumption~(G) by
applying the classical counterpart of Forster-Lyapunov criteria for conservative semigroups. Here, the conservative semigroup is the
one associated to the $h$-tranform of $P_n$ defined by $R_n f:=\frac{\theta_0^{-n}}{\eta}P_n(\eta f)$ (which is called $Q$-process in
the sub-Markovian case, cf. e.g.~\cite{MeleardVillemonais2012}). The only difficulty in the proof of the following theorem is that $\eta$ may vanish on some
subset of $E$.

\begin{thm}
  \label{thm:reciproque}
  Assume that there exist a positive function $\psi:E\rightarrow(0,+\infty)$
  and a non-negative eigenfunction $\eta\in L^{\infty}(\psi)$ of
  $P_1$ for the eigenvalue $\theta_0>0$, such that
  \begin{align}
    \label{equation3}
     \left|\theta_0^{-n} P_n f(x)-\eta(x)\nu_P(f)\right|\leq \zeta_n \psi(x)
  \end{align}
  is satisfied for all $x\in E$ and all measurable functions $f:E\rightarrow\mathbb{R}$
  such that $|f|\leq\psi$, where $(\zeta_n)_{n\geq 0}$ is some positive
  sequence converging to $0$. Then Assumption~(G) is satisfied with
  $\psi_2=\eta$ and with some function $\psi_1\in L^\infty(\psi)$ such
  that $\psi\in L^\infty(\psi_1)$.
\end{thm}

\begin{proof}
  We define $E'=\{x\in E,\ \eta(x)>0\}$ and
  introduce the conservative semigroup $R$ on functions $g:E'\rightarrow \R$ such
  that $|g(x)|\leq \psi(x)/\eta(x)$ defined by
  \[
    R_n g(x)=\frac{\theta_0^{-n}}{\eta(x)} P_n(\eta g)(x),\ \forall x\in E'\text{ and }n\geq 0.
  \]
  Applying~\eqref{equation3} to $f=g\eta$ and setting $\nu_R(dx)=\eta(x)\nu_P(dx)$, we deduce that, for all $x\in E'$ and all
  measurable function $g:E'\rightarrow\RR$ such that $|g|\leq \psi/\eta$
  \[
    \left|R_n g(x)-\nu_R(g)\right|\leq \zeta_n\frac{\psi(x)}{\eta(x)}.
  \]
  This is the classical $V$-uniform ergodicity condition (with $V=\psi/\eta$), for which necessary and sufficient conditions are
  well-known. First, it implies $V$-uniform geometric ergodicity, i.e. one can replace $\zeta_n$ by $C\,\beta^n$ for some
  $C>0,\beta\in(0,1)$ in the above equation (see for instance Proposition 15.2.3 in~\cite{DoucMoulinesEtAl2018}). Second, we deduce
  using for example Theorem~15.2.4(b) in~\cite{DoucMoulinesEtAl2018} that, for any integer $m$ such that $C^{1/m}\beta<1$ and any
  $\lambda,\rho$ such that $C^{1/m}\beta\leq \lambda <\rho<1$, there exist $d,C_R<+\infty$ such that
  \begin{equation}
    \label{eq:Lyapunov}
    R_1 V_0(x)\leq \rho V_0(x)+C_R\11_K(x),\quad\forall x\in E',
  \end{equation}
  with
  \[
  V_0=\sum_{k=0}^{m-1}\lambda^{-k} R_k\left(\frac{\psi}{\eta}\right)
  \]
  and $K:=\{\psi/\eta\leq d\}\cap E'$ is an accessible small set for
  $R$. This last property means that there exists a probability measure $\nu_R$ on
  $E'$ and a constant $c_R>0$ such that, for all $A\subset K$ measurable,
  \[
    R_{n'_1}\11_A(x)\geq c_R\nu_R(A),\quad\forall x\in K.
  \]
  for some constant integer $n'_1\geq 1$.  Since $K$ is accessible,
  there exists $n''_1\geq 0$ such that
  $a:=\nu_R R_{n''_1}\11_K>0$. Setting $n_1=n'_1+n''_1$, it then follows
  that
  \[
  P_{n_1}(\psi\11_A)(x)\geq c_R\theta_0^{n_1}\eta(x)\,\nu_R R_{n''_1}\left(\11_K\11_A\frac{\psi}{\eta}\right),\quad\forall x\in K.
  \]
  Due to the definition of $K$, we deduce that (G1) holds true with
  $c_1=ac_R\theta_0^{n_1}/d$ and the probability measure
  $\nu(dx)=\frac{\psi(x)}{a\eta(x)}\11_K(x) (\nu_R R_{n''_1})(dx)$.

  Defining $\psi_1=\eta V_0$, we also deduce from~\eqref{eq:Lyapunov} that,
  \[
  P_1\psi_1(x)\leq \theta_0\rho\psi_1(x)+C_R\11_K(x)\eta(x)\leq
  \theta_0\rho\psi_1(x)+\frac{C_R}{\sup_E |\eta|/\psi_1}\11_K(x)\psi_1(x),\quad\forall x\in E'. 
  \]
  In view of the definition of $V_0(x)$ for all $x\in E'$, we have
  \[
  \psi_1(x)=\sum_{k=0}^{m-1}(\lambda\theta_0)^{-k}P_k\psi(x),
  \]
  which also makes sense for $x\in E\setminus E'$. For such an $x$, we deduce from~\eqref{equation3} that $P_n\psi(x)\leq
  \zeta_{n}\theta_0^n\psi(x)$. Without loss of generality, increasing $m$, $\lambda$ and $\rho$ if necessary, we can assume that
  $\zeta_m^{1/m}\leq \lambda<\rho<1$.
  Then,
  \[
  P_1\psi_1(x)=\lambda\theta_0\psi_1(x)-\lambda\theta_0\psi(x)+(\lambda\theta_0)^{1-m} P_m\psi\leq
  \lambda\theta_0\psi_1(x),\quad\forall x\in E\setminus E'.
  \]
  Hence, we have checked that $P_1\psi_1\leq \theta_0\rho\psi_1+c_2\11_K\psi_1$ on $E$ for some constants $\rho<1$ and
  $c_2<+\infty$. Since $P_1\eta=\theta_0\eta$, the proof of (G2) is completed. Note also that $\psi\leq\psi_1$ and the fact that
  $\psi_1\in L^\infty(\psi)$ follows from the inequality $P_n\psi_1\leq A_n\psi_1$ for some constant $A_n$, which is an immediate
  consequence of~\eqref{equation3} and the fact that $\eta\in L^\infty(\psi_1)$.

  Thanks to Remark~\ref{rem:comments-hyp-G}, it is sufficient to check (G3) with $\psi_2=\eta$ instead of $\psi_1$. Since $\eta$ is
  an eigenfunction of $P_1$, (G3) is trivial.

  Since $K\subset E'$, it follows from~\eqref{equation3} that, for all $x\in K$, $\theta_0^{-n} P_n(\11_K\psi_1)(x)$ converges as
  $n\rightarrow+\infty$ to $\eta(x)\nu_P(\11_K\psi_1)>0$. Hence (G4) is clear. 
\end{proof}

For continuous time semigroups $(P_t)_{t\in[0,+\infty)}$, the
conclusions of Theorem~\ref{thm:main} can be easily deduced from
properties on the discrete skeleton $(P_{nt_0})_{n\in\N}$ (similar
properties where already observed in Theorem~5
of~\cite{TuominenTweedie1979} and
in~\cite{ChampagnatVillemonais2017b}). In the following result, the
function $\eta$ and the positive measure $\nu_P$ are the one of
Theorem~2.1 applied to the discrete skeleton $(P_{nt_0})_{n\in\N}$.

\begin{cor}
  \label{corollaryCont}
  Let $(P_t)_{t\in[0,+\infty)}$ be a continuous time semigroup. Assume
  that there exists $t_0>0$ such that $(P_{nt_0})_{n\in\N}$ satisfies
  Assumption~(G),
  $\left(\frac{P_t\psi_1}{\psi_1}\right)_{t\in[0,t_0]}$ is upper
  bounded by a constant $\bar c>0$ and
  $\left(\frac{P_t\psi_2}{\psi_2}\right)_{t\in[0,t_0]}$ is lower
  bounded by a constant $\underline c>0$. Then there exist some
  constants $C''>0$ and $\gamma>0$ such that, for all measurable
  functions $f:E\rightarrow\R$ satisfying $|f|\leq \psi_1$ and all
  positive measure $\mu$ on $E$ such that $\mu(\psi_1)<+\infty$ and
  $\mu(\psi_2)>0$,
  \begin{align}
    \label{equation1cont}
    \left|\frac{\mu P_t f}{\mu P_t \psi_1}-\nu_P(f)\right|\leq C'' e^{-\gamma t} \frac{\mu(\psi_1)}{\mu(\psi_2)},\quad\forall t\in[0,+\infty),
  \end{align}
  In addition, there exists $\lambda_0\in \R$ such that
  $\nu_P P_t=e^{\lambda_0 t}\nu_P$ for all $t\geq 0$, and $e^{-\lambda_0 t}P_t\psi_1$ converges
  uniformly and exponentially toward $\eta$ in
  $L^\infty(\psi_1)$ when $t\rightarrow+\infty$. Moreover, there exist some
  constants $C'''>0$ and $\gamma'>0$ such that, for all measurable functions
  $f:E\rightarrow\R$ satisfying $|f|\leq \psi_1$ and all positive
  measures $\mu$ on $E$ such that $\mu(\psi_1)<+\infty$,
  \begin{align}
    \label{equation2cont}
    \left|e^{-\lambda_0 t}\mu P_t f-\mu(\eta)\nu_P(f)\right|\leq C''' e^{-\gamma' t} \mu(\psi_1),\quad\forall t\in[0,+\infty).
  \end{align}
\end{cor}

\begin{rem}
  \label{rem:BCGM}
  In~\cite{BansayeCloezEtAl2019}, a similar result is obtained, but with the additional assumptions that $\psi_2>0$ on $E$ and
  $n_1=1$. In this restricted case, one can check using Remark~\ref{rem:comments-hyp-G} that their assumptions are equivalent to
  ours. The fact that $\psi_2$ can vanish allows to deal with non-irreducible semigroups (see~\cite[Section
  6]{ChampagnatVillemonais2017b}).
\end{rem}

\begin{rem}
  The adaptation of the counterpart of Theorem~\ref{thm:reciproque} to the countinuous-time setting is straightforward. A similar
  result was proven in~\cite{BansayeCloezEtAl2019}, where the authors assume in addition that $\zeta_n$ is geometrically decreasing
  and that $\eta$ is positive.
\end{rem}

\begin{proof}
Assuming without loss of generality that $t_0=1$ and applying~\eqref{equation1} to $\mu P_t$, where $t\in[0,1]$, and $f$
such that $\mu(\psi_1)<+\infty$ and $|f|\leq\psi_1$, one deduces that
  \[
  \left|\frac{\mu P_{t+n}f}{\mu P_{t+n}\psi_1}-\nu_P(f)\right|\leq C\alpha^{n} \frac{\mu P_t \psi_1}{\mu P_t \psi_2}\leq \frac{C\bar c }{\alpha\underline c}\alpha^{n+t}\frac{\mu(\psi_1)}{\mu(\psi_2)},
\]
which implies~\eqref{equation1cont}. Then, applying this inequality to
$\mu=\nu_P$ and letting $n$ go to infinity shows that
$\nu_P P_t f/\nu_P P_t \psi_1=\nu_P f$ for all $t\geq 0$. Choosing
$f=P_s\psi_1$ entails
$\nu_P P_{t+s}\psi_1=\nu_P P_t\psi_1\,\nu_P P_s\psi_1$ for all
$s,t\geq 0$, and hence $\nu_P P_t\psi_1=e^{\lambda_0 t}\nu_P \psi_1$
for all $t\geq 0$ for some constant $\lambda_0\in\R$ (note that
$\theta_0=e^{\lambda_0}$).

  Similarly, inequality~\eqref{equation2} applied to $\mu=\delta_x P_t$ and $f=\psi_1$ on the one hand and to $\mu=\delta_x$ and
  $f=P_t \psi_1$ on the other hand implies that $P_t\eta(x)=\eta(x)\nu_P(P_t \psi_1)=e^{\lambda_0 t} \eta(x)$ for all $t\geq 0$.
  Applying again~\eqref{equation2} to $\mu=\delta_x P_t$ entails that
  \[
  \left|\theta_0^{-n}P_{t+n} f(x)-P_t\eta(x)\nu_P(f) \right|\leq C'\beta^nP_t\psi_1(x)\leq \frac{C'\bar c}{\beta}\beta^{n+t} \psi_1(x).
  \]
  In particular, for all $t\geq 0$,
  \[
  \left|e^{-\lambda_0 t}P_{t} f(x)-\eta(x)\nu_P(f) \right|\leq \frac{C'\bar c}{\beta}\beta^{t} \psi_1(x)
  \]
  and $e^{-\lambda_0 t}P_t\psi_1$ converges geometrically to $\eta$ in $L^\infty(\psi_1)$. This concludes the proof of Corollary~\ref{corollaryCont}
\end{proof}

\section{Some applications}
\label{sec:applications}

Given a positive semigroup $P$ acting on measurable functions on $E$,
one can try to directly check Assumption~(G) by finding appropriate
functions $\psi_1$ and $\psi_2$. Another natural and
equivalent strategy is to find a function $\psi$ such that the
semigroup defined by $Q_n f=\frac{P_n(\psi f)}{c^n\psi}$ is
sub-Markovian and check that it satisfies Assumption~(E)
of~\cite{ChampagnatVillemonais2017b}. The main advantage of this last
approach is that $Q$ has a probabilistic interpretation as the
semigroup of a sub-Markov process. As such, one can apply all the
criteria developed in the above mentioned reference and, more
generally, use the intuitions and toolboxes of the theory of
stochastic processes. Since both approaches are equivalent, this is
rather a question of taste.

In Subsection~\ref{sec:perturbedDyn}, we consider the case of a penalized perturbed
dynamical system and check Assumption~(G) directly. In subsection~\ref{sec:diffusionProc}, we consider the case of a
penalized diffusion processes and check Assumption~(E).

\subsection{Perturbed dynamical systems}
\label{sec:perturbedDyn}

Let $F:\R^d\rightarrow\R^d$ be a locally bounded measurable function and consider the perturbed dynamical system $X_{n+1}=F(X_n)+\xi_n$ with
$(\xi_i)_{i\in\ZZ_+}$ i.i.d.~non-degenerate Gaussian random variables. We are interested in the asymptotic behaviour of the
associated Feynman-Kac semigroup 
\begin{align*}
  P_n f(x)=\E_x\left(\prod_{k=1}^{n} G(X_k)\11_{X_k\in E} f(X_n)\right),
\end{align*}
where $E$ is a measurable subset of $\R^d$ with positive Lebesgue
measure and $G:E\rightarrow (0,+\infty)$ is a measurable function.

\begin{prop}
  \label{prop:SDP}
  Assume that $1/G$ is locally bounded, $G(x)\leq C\,\exp(|x|)$ for all
  $x\in E$ and some constant $C>0$ and there exists $p>1$ such that $|x|-p|F(x)|\rightarrow+\infty$ when $|x|\rightarrow+\infty$,
  then the semigroup $(P_n)_{n\in\N}$ satisfies Assumption~(G).
\end{prop}

\begin{proof}
  One easily checks that $\psi_1(x)=\exp(a|x|)$, where $a>0$ is such that $1/a<p-1$, satisfies
  \begin{equation}
    \label{eq:preuve-SDP}
    P_1\psi_1(x)\leq C\EE\left(e^{(1+a)|F(x)+\xi_1|}\right)\leq C'\,\psi_1(x)\,\exp\left(-a\left(|x|-p|F(x)|\right)\right),
  \end{equation}
  where $C'=C\EE e^{(1+a)|\xi_1|}$. Now, assume without loss of generality that $B(0,1)\cap E$ has positive Lebesgue measure and set
  $\theta_2:=\inf_{x\in B(0,1)\cap E} P_1\11_{B(0,1)\cap E} (x)/2$, which is clearly positive. It then follows from Markov's property
  that
  \[
  \theta_2^{-n}\inf_{x\in B(0,1)\cap E} P_n\11_{B(0,1)\cap E}(x)\geq \theta_2^{-n}\inf_{x\in B(0,1)\cap E} \EE_x\left[\prod_{k=1}^n
    G(X_k)\11_{B(0,1)\cap E}(X_k)\right]\geq 2^n\rightarrow+\infty,
  \]
  when $n\rightarrow+\infty$. One easily deduces that, for all $R\geq 1$, $\theta_2^{-n}\inf_{x\in B(0,R)\cap E} P_n\11_{B(0,1)\cap
    E}(x)\rightarrow+\infty$, and hence that $\theta_2^{-n}\inf_{x\in B(0,R)\cap E} P_n\11_{B(0,R)\cap
    E}(x)\rightarrow+\infty$ when $n\rightarrow+\infty$.

  We set $\theta_1=\theta_2/2$ and fix $R\geq 1$ large enough so that $C' e^{-a(|x|-p|F(x)|)}\leq\theta_1$ for all $|x|\geq R$. It
  then follows from~\eqref{eq:preuve-SDP} that $P_1\psi_1\leq \theta_1 \psi_1+c_2\11_K \psi_1$, where $K:=B(0,R)\cap E$. Setting
  $\psi_2(x)=\sum_{k=0}^{n_0}\theta_2^{-k}P_k\11_K(x)$, we deduce that, for all $x\in E$,
  \[
  P_1\psi_2(x)=\sum_{k=0}^{n_0}\theta_2^{-k}P_{k+1}\11_K(x)=\theta_2\psi_2(x)+\theta_2\left[\theta_2^{-(n_0+1)}P_{n_0+1}\11_K(x)-\11_K(x)\right]\geq \theta_2\psi_2(x)
  \]
  for $n_0$ chosen large enough. Since in addition $P_k\11_K\leq P_k\psi_1\leq (\theta_1+c_2)^k\psi_1$, $\psi_2\in L^\infty(\psi_1)$
  and, for all $x\in K$, $\psi_2(x)\geq 1\geq e^{-aR}\psi_1(x)$. Hence, dividing $\psi_2$ by $\|\psi_2/\psi_1\|_{\infty}$ ends the
  proof of (G2).

  In order to prove~(G1), (G3) and~(G4), we follow similar arguments as for~\cite[Prop. 7.2]{ChampagnatVillemonais2017b}. Since the
  adaptation of these arguments is a bit tricky because the function $\psi_1$ needs to be taken into account appropriately, we give
  the details below.

  The first step consists in proving that there exists a constant $c>0$ such that, for all measurable $A\subset K$, for all $x\in E$
  and all $y\in K$,
  \begin{equation}
    \label{eq:step1}
    \frac{P_1(\psi_1\11_A)(x)}{\psi_1(x)}\leq c\frac{P_1(\psi_1\11_A)(y)}{\psi_1(y)}.  
  \end{equation}
  This implies easily (G1) for $n_1=1$ and (G4) then follows directly from (G1) (since $n_1=1$).
  
  To prove~\eqref{eq:step1}, we observe that (recall that $A\subset K=E\cap B(0,R)$)
  \[
  \frac{P_1(\psi_1\11_A)(x)}{\psi_1(x)}\leq P_1(\psi_1\11_A)(x) \leq \sup_{|z|\leq R}[G(z)\psi_1(z)]\ \PP(F(x)+\xi_1\in E\cap A\cap B(0,R)).
  \]
  Because $\xi_1$ is a non-degenerate gaussian random variable, it is not hard to check that there exists a constant $C_R$ depending
  only on $R$ (and not on $x\in E$ and $y\in K$) such that $\PP(F(x)+\xi_1\in E\cap A\cap B(0,R))\leq C_R \PP(F(y)+\xi_1\in E\cap
  A\cap B(0,R))$. Therefore,
  \[
  \frac{P_1(\psi_1\11_A)(x)}{\psi_1(x)}\leq C_R \frac{\sup_{|z|\leq R}G(z)\psi_1(z)}{\inf_{|z|\leq R}
    G(z)}\EE_y\left[G(X_1)\psi_1(X_1)\11_{X_1\in E\cap A}\right]\leq c\frac{P_1(\psi_1\11_A)(y)}{\psi_1(y)},
  \]
  where $c=C_R e^{aR}\sup_{|z|\leq R}G(z)\psi_1(z)/\inf_{|z|\leq R} G(z)$. Hence~\eqref{eq:step1} is proved.

  Next, we observe that the Markov property combined with (G2) implies that, for all $x\in E$ and all $n\geq 1$,
  \begin{equation}
    \label{eq:step2}
    \EE_x\left[\prod_{k=1}^n G(X_k)\11_{X_k\in E\setminus K}\psi_1(X_n)\right]\leq (\theta_1+c_2)\theta_1^{n-1} \psi_1(x).
  \end{equation}
  We also have the property that there exists a constant $c'>0$ such that, for all $y\in K$ and all $0\leq k\leq n$,
  \begin{equation}
    \label{eq:step3}
    \frac{P_n\psi_1(y)}{\psi_1(y)}\geq c'\theta_2^k\frac{P_{n-k}\psi_1(y)}{\psi_1(y)}.
  \end{equation}
  As observed in Remark~\ref{rem:comments-hyp-G}, since we already
  proved (G2), the last property is equivalent to the same one with
  $\psi_2$ instead of $\psi_1$.
  Since $P_1\psi_2\geq \theta_2\psi_2$ on $K$~\eqref{eq:step3} is then clear.

  The proof of~(G3) can then be done by combining the last inequalities. We first decompose $P_n\psi_1$ depending on the value of the
  first return time in $K$: for all $x\in E$,
  \begin{align*}
    P_n\psi_1(x) & =\EE_x\left[\prod_{k=1}^n G(X_k)\11_{X_k\in E\setminus
        K}\psi_1(X_n)\right]+\sum_{\ell=1}^n\EE_x\left[\prod_{k=1}^{\ell-1}G(X_k)\11_{X_k\in E\setminus K}G(X_\ell)\11_{X_\ell\in
        K}P_{n-\ell}\psi_1(X_\ell)\right] \\
    & \leq (\theta_1+c_2)\theta_1^{n-1}\psi_1(x)+\sum_{\ell=1}^n\EE_x\left[\prod_{k=1}^{\ell-1}G(X_k)\11_{X_k\in E\setminus K}\EE_{X_{\ell-1}}\left[G(X_1)\11_{X_1\in
        K}P_{n-\ell}\psi_1(X_1)\right]\right], 
  \end{align*}
  where we used~\eqref{eq:step2} and Markov's property in the second line. We then proceed by using~\eqref{eq:step1} for some fixed $y\in K$
  first,~\eqref{eq:step2} next, and finally~\eqref{eq:step3} twice:
  \begin{align*}
    \frac{P_n\psi_1(x)}{\psi_1(x)} & \leq
    (\theta_1+c_2)\theta_1^{n-1}+\frac{c}{\psi_1(x)}\sum_{\ell=1}^n\EE_x\left[\prod_{k=1}^{\ell-1}G(X_k)\11_{X_k\in E\setminus
        K}\psi_1(X_{\ell-1})\right]\frac{\EE_{y}\left[G(X_1)\11_{X_1\in
          K}P_{n-\ell}\psi_1(X_1)\right]}{\psi_1(y)} \\
    & \leq \frac{\theta_1+c_2}{\theta_1}\theta_1^{n}+\frac{c(\theta_1+c_2)}{\theta_1}\sum_{\ell=1}^n
    \theta_1^{\ell-1}\frac{P_{n-\ell+1}\psi_1(y)}{\psi_1(y)} \\ 
    & \leq \left[\frac{\theta_1+c_2}{c'\theta_1}\left(\frac{\theta_1}{\theta_2}\right)^{n}+\frac{c(\theta_1+c_2)}{c'\theta_1}\sum_{\ell=1}^n
    \left(\frac{\theta_1}{\theta_2}\right)^{\ell-1}\right]\,\frac{P_{n}\psi_1(y)}{\psi_1(y)}.
  \end{align*}
  Since the last factor is bounded in $n$, this ends the proof of Proposition~\ref{prop:SDP}.
\end{proof}

\subsection{Diffusion processes}
\label{sec:diffusionProc}

Let $(X_t)_{t\in [0,+\infty)}$ be solution to the SDE
\begin{equation}
  \label{eq:SDE}
  {\rm d} X_t={\rm d}B_t+b(X_t)\,{\rm d}t,\quad X_0\in(0,+\infty)^d,  
\end{equation}
where $B=(B^{(1)},\ldots,B^{(d)})$ is a standard $d$-dimensional Brownian motion and
$b:\R^d\rightarrow\R^d$ is locally H\"older. Let
$r:(0,+\infty)^d\rightarrow \R$ be locally bounded and consider the
semigroup $(P_t)_{t\in [0,+\infty)}$ defined by
\begin{equation}
  \label{eq:sg-diffusion}
  P_t f(x)=\E_x\left(e^{\int_0^t r(X_u)\,{\rm d}u}\,f(X_t)\,\11_{X_s\in(0,+\infty)^d,\ \forall s\in[0,t]}\right).  
\end{equation}
The term $\11_{X_s\in(0,+\infty)^d,\ \forall s\in[0,t]}$ above corresponds to a killing at the boundary of the domain $(0,+\infty)^d$.
Note that the solution to~\eqref{eq:SDE} may explode in finite time if $b$ does not satisfy the linear growth condition. However, we
assume by convention that $X_t\not\in(0,+\infty)^d$ after the explosion time, so that~\eqref{eq:sg-diffusion} makes sense. We refer
to~\cite[Sections 4.1 and 12.1]{ChampagnatVillemonais2017b} for the precise construction of the process.

One motivation for the study of this semigroup comes from the Feynam-Kac formula. Indeed, when the coefficients are smooth enough
(see for instance~\cite[Section~1.3.3]{Pham2009}), this semigroup is solution to the Cauchy linear parabolic partial differential
equation
\begin{align*}
  &rv-\frac{\d v}{\d t}+\cL v=0,\text{ on }[0,+\infty)\times(0,+\infty)^d\\
  &v(0,\cdot)=f,\text{ on }(0,+\infty)^d,
\end{align*}
where $\cL$ is the differential operator of second order
\[
\cL\varphi(x)=\frac{1}{2}\Delta \varphi(x)+b(x)\cdot\nabla\varphi(x),\quad \forall \varphi\in C^2(\R^d),
\]
with Dirichlet boundary conditions.

\begin{prop}
  Assume that
  \begin{equation}
    \label{eq:hyp-diffusion}
    r(x)+\sum_{i=1}^d\,b_i(x)\xrightarrow[|x|\rightarrow\infty,\ x\in(0,\infty)^d]{} -\infty.    
  \end{equation}
  Then the semigroup $(P_t)_{t\in [0,+\infty)}$ satisfies the assumptions of Corollary~\ref{corollaryCont}.
\end{prop}

\begin{proof}
  We first observe that, setting $\psi(x)=\exp\left(\sum_{i=1}^d x_i \right)$ and $a:=d/2+\sup_{x\in(0,\infty)^d} r(x)+\sum_{i=1}^db_i(x)$, we have, for all $x\in(0,+\infty)$,
  \[
  Q_t f(x):=e^{-at}\frac{P_t(f\psi)(x)}{\psi(x)}=\E_x\left(e^{-\frac{d}{2}t+\sum_{i=1}^d B^{(i)}_t}\,e^{\int_0^t \left(r(X_u)+\sum_{i=1}^d b_i(X_u) -a+\frac{d}{2}\right)\,{\rm d}u}
    \,f(X_t)\,\11_{X_s\in(0,+\infty)^d,\ \forall s\in[0,t]}\right).
  \]
  Using Girsanov's theorem, we deduce that
  \[
  Q_t f(x)=\E_x\left(e^{-\int_0^t \kappa(\bar X_u)\,{\rm d}u}\,f(\bar X_t)\,\11_{\bar X_s\in(0,+\infty)^d,\ \forall s\in[0,t]}\right).
  \]
  where $\kappa(y)=a-r(y)-\frac{d}{2}-\sum_{i=1}^d b_i(y)\geq 0$ and $\bar X=( \bar{X}^{(1)}, \ldots, \bar{X}^{(d)} )$ is solution to the SDE ${\rm d}\bar X^{(i)}_t= {\rm d}
  B^{(i)}_t+(1+b_i(\bar X_t))\,{\rm d}t$ with $\bar{X}^{(i)}_0=x_i$.
  
  Assumption~\eqref{eq:hyp-diffusion} thus
  implies that the conditions of
  \cite[Theorem~4.5]{ChampagnatVillemonais2017b} are
  satified\footnote{To prove (4.12) therein, one can use the same
    argument as the one used in Corollary 4.3 of this reference.} and
  hence that $Q$ satisfies Assumption~(F) therein, which implies that
  Assumption (E) is satisfied by the semigroup $Q_{nt_0}$ for some
  $t_0>0$ and some Lyapunov functions $\varphi_1$ and $\varphi_2$,
  that $\left(\frac{Q_t\varphi_1}{\varphi_1}\right)_{t\in[0,t_0]}$ is
  uniformly bounded, and that there exist a positive function
  $\eta_Q\in L^{\infty}(\varphi_1)$ and a constant $\lambda_0>0$ such
  that $Q_t\eta_Q=e^{-\lambda_0 t}\eta_Q$ for all
  $t\in[0,+\infty)$.

  To conclude, it remains to observe that the same procedure as the
  one used in the proof of Theorem~\ref{thm:main} above allows to
  deduce from these properties that $(P_{nt_0})_{n\geq 0}$ satisfies
  Assumption (G) with $\psi_1=\psi\varphi_1$ and
  $\psi_2=\psi\eta_Q$. Observing also that $\psi_2$ is the function $\eta$ of
  Theorem~\ref{thm:main}, we deduce that $(P_t)_{t\in[0,+\infty)}$
satisfies the assumptions of Corollary~\ref{corollaryCont}.
\end{proof}

\bibliographystyle{abbrv}
\bibliography{biblio-bio,biblio-denis,biblio-math}

\end{document}